\documentclass[11pt]{amsart}

\usepackage{xypic}
\usepackage{amssymb}

\newtheorem{theorem}{Theorem}[section]
\newtheorem{lemma}[theorem]{Lemma}
\newtheorem{corollary}[theorem]{Corollary}
\newtheorem{proposition}[theorem]{Proposition}
\theoremstyle{remark}
\newtheorem{remark}[theorem]{Remark}
\theoremstyle{remark}
\newtheorem{example}[theorem]{Example}

\newcommand\md{{\text{\rm mod}}} 
\newcommand\Ab{{{\mathcal A}b}} 
\newcommand\Md{{\text{\rm Mod}}} 

\newcommand{\Hom}{\mathop{\rm Hom}\nolimits} 
\newcommand{\Prod}{\mathop{\rm Prod}\nolimits} 
\newcommand{\Ker}{\mathop{\rm Ker}\nolimits} 
\newcommand\lex{\mathop{\rm Lex}\nolimits} 
\newcommand\colim{\mathop{\rm colim}} 
\newcommand{\Inj}{\mathop{\rm Inj}\nolimits}
\newcommand{\card}{\mathop{\rm card}\nolimits} 
\newcommand{\A}{\mathop{\rm A}\nolimits} 
\newcommand{\HK}{\mathop{\rm\bf K}\nolimits} 

\newcommand\opp{{^{\text{\rm op}}}} 

\newcommand\CA{{\mathcal A}}

\newcommand\CC{{\mathcal C}}
\newcommand\DC{{\mathcal D}}

\newcommand\CN{{\mathcal N}}

\newcommand\CS{{\mathcal S}}
\newcommand\CT{{\mathcal T}}

\newcommand\GR{{\mathfrak R}} 

\newcommand\N{{\mathbb N}} 
\newcommand\Z{{\mathbb Z}} 

\begin{document}

\title[A representability theorem]{A representability theorem for some huge abelian
categories}

\author{George Ciprian Modoi}
\thanks{Research supported by CNCS-UEFISCDI grant PN-II-RU-TE-2011-3-0065}

\address{"Babe\c s-Bolyai" University\\ RO-400084, Cluj-Napoca
\\ Romania}

\email[George Ciprian Modoi]{cmodoi@math.ubbcluj.ro}



\begin{abstract}
We define quasi--locally presentable categories as big unions of
a chain of coreflective subcategories which are locally presentable. Under
appropriate hypotheses we prove a representability theorem for
exact contravariant functors defined on a quasi--locally
presentable category taking values in abelian groups. We show that
the abelianization of a well generated triangulated category is
quasi--locally presentable and we obtain a new proof of Brown
representability theorem. Examples of functors which are not
representable are also given.
\end{abstract}
\maketitle

\section*{Introduction}
One of the main problems occurring in the theory of triangulated
categories is to construct a left or right adjoint for a given
triangulated functor. In his influential book on this subject,
Neeman shows that the problem of finding an adjoint for a functor
between triangulated categories may be equivalently studied at the
level of abelianizations of these categories, where we have to
construct an adjoint for some exact functor between abelian
categories (see \cite[Proposition 5.3.9]{N}). Further Neeman
considers in \cite[Remark 5.3.10]{N} that, unfortunately this idea
is ``nearly impossible'' to be applied, since ``existence theorems
of adjoints usually depend on the categories being
well--powered'', that is one object must have only a set of
subobjects (for an object of an abelian category this it
equivalent to having only a set of quotients). But, in general,
the abelianization of a triangulated category with arbitrary
coproducts is huge, that is it does not satisfy the condition of
being well (co)powered; see \cite[Appendix C]{N}. Hence the
abelianization is often considered to be too big, hence not
manageable (see also the Introduction of Krause's work \cite{KL}).
This paper intends to change a little this perspective. More
exactly, the result about the existence of adjoints depending on
the categories being well powered is, obviously, the special
Freyd's adjoint functor theorem: if $\CC$ is a complete, well
powered category having a cogenerator, then every functor
$F:\CC\to\DC$ has a left adjoint if and only if it preserves
limits, see \cite[p. 89]{PF}. We argue that even if the abelianization of a well
generated triangulated category is not always well (co)powered, it
has enough structure allowing us to apply the general Freyd's
adjoint functor theorem: if $\CC$ is a complete category, then
every functor $F:\CC\to\DC$ has a left adjoint if and only if it
preserves limits and satisfies the solution set condition (that is
for every $f:D\in\DC$ there is a set maps $f_i:D\to F(C_i)$, $i\in I$ in $\DC$,
where $C_i\in\CC$, such that every map $D\to
F(C)$, with $C\in\CC$, factors as $f=F(k)f_i$, for some 
$k:C_i\to C$ in $\CC$; see
\cite[0.7]{AR}). The problem of the existence of the adjoints and
the one of representability of a given functor are strongly
related (to fix the settings, suppose that we work with
preadditive categories): First, a functor $F:\CC\to\DC$ has a left
adjoint if and only if the functor $\DC(D,F(-)):\CC\to\Ab$ is
representable for all $D\in\DC$; second a functor $F:\CC\to\Ab$
has a left adjoint if and only if it is representable (actually it
is represented by the left adjoint evaluated at $\Z$, see \cite[p. 81-82]{PF}).

The paper is organized as follows: In the first section we
introduce the notion of quasi--locally presentable category; it is
a category which may be written as a union of a chain of coreflective
subcategories which are locally $\lambda$--presentable, where
$\lambda$ runs over all regular cardinals. Under appropriate
hypotheses, we prove a representability theorem for exact,
contravariant functors defined on such categories.

In the second section we recall the definition of the
abelianization of a triangulated category and we show how the
study of Brown representability may be done at the level of this
abelianization. For well generated triangulated categories we show
that the abelianization is quasi--locally presentable and
satisfies the supplementary hypotheses allowing us to apply the
representability theorem proved in the previous section. As a
consequence we obtain a new proof of Brown representability
theorem for well generated triangulated categories.

All categories which we work with are preadditive (enriched over
$\Ab$). Everywhere in our paper we may equally adopt the point of
view of G\"odel--Bernays--Von Neumann axiomatization of set
theory, with the distinction made there between classes and sets,
or to work in a given Grothendieck universe. In this last case, a
set means a small set relative to that universe, whereas a class
is a set which is not necessarily small.

\vskip5mm\noindent{\em Acknowledgements.} For the second an third version
of this paper we acknowledge the financial support of the grant
CNCS-UEFISCDI code PN-II-RU-TE-2011-3-0065. We also would like to
thank an anonymous referee for many suggestions of improvement of
the earlier version.

\section{Quasi--locally presentable abelian categories}

We begin this section by recalling some definitions: A cardinal
$\lambda$ is said to be {\em regular\/} provided that it is
infinite and it cannot be written as a sum of less than $\lambda$
cardinals, all smaller than $\lambda$. Denote by $\GR$ the class
of all regular cardinals.

Let $\CA$ be a additive category and $\CC\subseteq\CA$ be a
subcategory. Let $F:\CA\to\Ab$ be a contravariant functor. The
{\em category of elements\/} of $F|_\CC$, where $F|_\CC$ denotes
the restriction of $F$ at $\CC$, is by definition constructed as
follows:
\[\CC/F=\{(X,x)\mid X\in\CC,x\in F(C)\},\] with the morphisms
\[\CC/F((X_1,x_1),(X_2,x_2))=\{\alpha\in\CC(X_1,X_2)\mid
F(\alpha)(x_2)=x_1\}.\] In particular, for any object $A\in\CA$,
let denote
\[\CC/A=\CC/{\CA(-,A)}=\{(C,\xi)\mid C\in\CC,\xi:C\to A\},\]
\[\CC/A((C_1,\xi_1),(C_2,\xi_2))=\{\alpha\in\CC(C_1,C_2)\mid
\xi_2\alpha=\xi_1\}.\]

Consider a regular cardinal $\lambda$. A non--empty category $\CS$
is called $\lambda${\em--filtered} if the following two conditions
are satisfied: \begin{itemize}\item[F1.] For every set $\{s_i\mid
i\in I\}$ of less that $\lambda$ objects of $\CS$ there are an
object $s\in\CS$ and morphisms $s_i\to s$ in $\CS$, for all $i\in
I$. \item[F2.] For every set $\{\sigma_i:s\to t\mid i\in I\}$ of
less that $\lambda$ morphisms in $\CS$, there is a morphism
$\tau:t\to u$ such that $\tau\sigma_i=\tau\sigma_j$, for all
$i,j\in I$.
\end{itemize}
Let $A$ be an object of a category $\CA$. Then the functor
$\CA(A,-)$ preserves  the colimit of a diagram $\CS\to\CA,\
s\mapsto X(s)$ in $\CA$ (indexed over a category $\CS$), if and
only if every map $g:A\to\colim_{s\in\CS}X(s)$ factors as
\[\diagram A\dto_{f}\drto^{g}&\\
X(u)\rto_{\xi_u\hskip5mm}&\colim_{s\in\CS}X(s)\enddiagram\]
through some of the canonical maps $\xi_u$ with $u\in\CS$, and every such
factorization is essentially unique, in the sense that if
$f_1,f_2:A\to X(u)$ with $\xi_uf_1=g=\xi_uf_2$ then there is
$\sigma:u\to t$ a map in $\CS$ such that
$X(\sigma)f_1=X(\sigma)f_2$. The object $A\in\CA$ is called
{\em$\lambda$--presentable} if $\CA(A,-)$ preserves all
$\lambda$--filtered colimits. The category $\CA$ is called {\em
locally $\lambda$--presentable} provided that it is cocomplete,
and has a set $\CS$ of $\lambda$--presentable objects such that
every $X\in\CA$ is a $\lambda$--filtered colimit of objects in
$\CS$ (see \cite[Definition 1.17]{AR}, but also \cite[Remark
1.21]{AR})). Note that, if $\CA$ is locally
$\lambda$--presentable, then the subcategory $\CA^\lambda$ of all
$\lambda$--presentable objects in $\CA$ is essentially small, and
for every object $A\in\CA$, the category $\CA^\lambda/A$ is
$\lambda$--filtered and
\[A\cong\colim_{(X,\xi)\in\CA^\lambda/A}X,\] as we may see from
\cite[Proposition 1.22]{AR}. A category is called {\em locally
presentable} if it is locally $\lambda$--presentable for some
regular cardinal $\lambda$.

\begin{remark}\label{fsaft} Let $\CA$ be a locally $\lambda$--presentable
category. Observe than the category $\CA\opp$ satisfies the hypotheses of Freyd's special adjoint functor theorem:
it is well powered, complete and has a cogenerator
(since the coproduct of all $\lambda$-presentable objects is a generator for $\CA$).
In particular, every contravariant functor $F:\CA\to\Ab$ which
sends colimits into limits is representable. Indeed, we can view $F$ as a covariant functor  
$\CA^\opp\to\Ab$ which must be representable, having a left adjoint.
Let us write $F\cong\CA(-,A)$, for some $A\in\CA$. Thus the categories $\CA^\lambda/A$ and
$\CA^\lambda/F$ are isomorphic, so
\[F\cong\CA(-,\colim_{(X,x)\in\CA^\lambda/F}X).\]
\end{remark}

We consider a category $\CA$ which is
a union \[\CA=\bigcup_{\lambda\in\GR}\CA_\lambda,\] of a chain of subcategories 
$\{\CA_\lambda\mid\lambda\in\GR\}$
such that $\CA_\kappa\subseteq\CA_\lambda$ for all $\kappa\leq\lambda$ and the
subcategory $\CA_\lambda$ locally $\lambda$--presentable and
closed under colimits in $\CA$, for any $\lambda\in\GR$.  
Denote by $I_\lambda:\CA_\lambda\to\CA$ the
inclusion functor, which preserves colimits by our assumption.
Note that by Freyd's special adjoint functor theorem, the
subcategory $\CA_\lambda$ is coreflective, that is $I_\lambda$ has
a right adjoint $R_\lambda:\CA\to\CA_\lambda$. We call {\em quasi--locally presentable} 
a category $\CA$ as above satisfying the additional propperty that $R_\lambda$ preserves colimits 
for all $\lambda\in\GR$. For such a
quasi--locally presentable category $\CA$ and a regular cardinal
$\lambda$ we denote by $\CA_\lambda^\lambda$ the subcategory of
all $\lambda$--presentable objects of $\CA_\lambda$, which has to
be skeletally small. 

\begin{lemma}\label{akkinall} In a quasi--locally presentable category $\CA$ it holds 
$\CA_\kappa^\kappa\subseteq\CA_\lambda^\lambda$, for every  $\kappa\leq\lambda$.
\end{lemma}

\begin{proof}
With the notations above, fix two cardinals $\kappa\leq\lambda$.
Observe that if we denote $I_{\kappa,\lambda}:\CA_\kappa\to\CA_\lambda$ the inclusion functor, 
then it has a right adjoint namely $R_{\kappa,\lambda}=R_\kappa I_\lambda$.
Since $R_\kappa$ preserves colimits, $R_{\kappa,\lambda}$ 
satisfies the same property. Then for $A\in\CA_\kappa^\kappa$ and 
for a $\lambda$-filtered (hence also $\kappa$-filtered) diagram $(X_i)_{i\in I}$ in 
$\CA_\lambda$ we have the following chain of isomorphisms, showing that 
$I_{\kappa,\lambda}(A)$ is $\lambda$-presentable:

\begin{align*}
 \CA_\lambda(I_{\kappa,\lambda}(A),\colim X_i)&\cong\CA_\kappa(A,R_{\kappa,\lambda}(\colim X_i))
 \cong\CA_\kappa(A,\colim R_{\kappa,\lambda}(X_i))\\
  &\cong\colim\CA_\kappa(A,R_{\kappa,\lambda}(X_i))\cong\colim\CA_\lambda(I_{\kappa,\lambda}(A),X_i).
\end{align*}
\end{proof}

As an example of quasi--locally presentable categories we mention
first the classical locally presentable ones. Clearly if $\CA$ is
locally $\kappa$--presentable for some regular cardinal $\kappa$,
then it is also quasi--locally presentable, for all regular cardinals $\lambda$ putting
$\CA_\lambda=\CA$, if $\lambda\geq\kappa$, and $\CA_\lambda=0$ otherwise.

\begin{lemma}\label{fiscolimit}
Let $F:\CA\to\Ab$ be a contravariant functor which sends colimits
into limits, defined on a quasi--locally presentable, abelian
category $\CA$. Then for every regular cardinal $\kappa$, there is
$\lambda\in\GR$, $\lambda\geq\kappa$ such that
\[FI_\kappa\cong\colim_{(X,x)\in\CA_\lambda^\lambda/F}\CA(I_\kappa(-),X).\]
\end{lemma}

\begin{proof} For any $\lambda\in\GR$, consider the
corresponding coreflective locally $\lambda$--presentable
subcategory $I_\lambda:\CA_\lambda\leftrightarrows\CA:R_\lambda$.

Fix $\kappa\in\GR$. For a skeleton $\CC_0$ of $\CA_\kappa^\kappa$,
denote $C_0=\coprod_{(U,u)\in\CC_0/F}U$. Let $\lambda$ be a
regular cardinal such that
\[\lambda>\kappa+\card\CC_0+\sum_{U\in\CC_0}\card
F(U)+\sum_{U\in\CC_0}\card\CA(U,C_0)+\aleph_1.\]

Since  $F:\CA\to\Ab$ sends colimits into limits, the same property
is also true for $FI_\lambda:\CA_\lambda\to\Ab$. By Remark
\ref{fsaft} we obtain $FI_\lambda\cong\CA_\lambda(-,F_\lambda)$
for some $F_\lambda\in\CA_\lambda$ satisfying
\[F_\lambda=\colim_{(X,x)\in\CA_\lambda^\lambda/F}X=\colim_{(X,\xi)\in\CA_\lambda^\lambda/F_\lambda}X,\]
with the canonical maps $\gamma_{(X,x)}:X\to F_\lambda$. Note that $\gamma_{(X,x)}$ 
is the image of $(X,x)$ via the isomorphism of categories 
$\CA^\lambda_\lambda/FI\lambda\stackrel{\cong}\longrightarrow\CA^\lambda_\lambda/F_\lambda$.
 We have to show that
\[F(A)\cong\colim_{(X,x)\in\CA_\lambda^\lambda/F}\CA\left(A,X\right),\]
for all $A\in\CA_\kappa$. Since $A=I_\kappa(A)=I_\lambda(A)$ this
means precisely that $\CA(A,-)$ preservers the colimit of the diagram
$\CA_\lambda^\lambda/F\to\CA,\ (X,x)\mapsto X$. In order to prove
this, consider in the first step that $A$ is a coproduct of
objects in $\CA_\kappa^\kappa$. Without losing the generality we
may assume that $A=\coprod_{i\in I}U_i$, for some set $I$, and
some $U_i\in\CC_0$. Denote by $j_i:U_i\to A,\ (i\in I)$ the
canonical injections. Let $g:A\to F_\lambda$ be a map in $\CA$.
Since for all $U\in\CC_0$ we have
$U\in\CA_\kappa\subseteq\CA_\lambda$, we may identify $\CC_0/F$
with $\CC_0/F_\lambda$ thus
$C_0=\coprod_{(U,\upsilon)\in\CC_0/F_\lambda}U$ with the canonical
injections $\epsilon_{(U,\upsilon)}:U\to C_0$. Since
$gj_i\in\CA(U_i,F_\lambda)$ we get a unique $f:A\to C_0$, such
that $fj_i=\epsilon_{(U_i,gj_i)}$ from the universal property of
the coproduct. Put
$c_0=(\upsilon)_{(U,\upsilon)\in\CC_0/F_\lambda}$. We know by Lemma \ref{akkinall} 
that $\CA_\kappa^\kappa\subseteq\CA_\lambda^\lambda$, so the condition 
$\lambda>\sum_{U\in\CC_0}\card F(U)$ assures us that
$(C_0,c_0)\in\CA^\lambda_\lambda/F_\lambda$. It follows
$(C_0,c_0)\in\CA_\lambda^\lambda/F$. Moreover by
construction $\gamma_{(C_0,c_0)}f=g$, so $g$ factors through
$\gamma_{(C_0,c_0)}$.

It remains to show that this factorization is essentially unique.
Consider therefore two maps $f_1,f_2:A\to C_0$ such that
$\gamma_{(C_0,c_0)}f_1=g=\gamma_{(C_0,c_0)}f_2$. Denote
$\CN=\{(U,h)\mid U\in\CC_0,h\in\CA(U,C_0)\hbox{ with
}\gamma_{(C_0,c_0)}h=0\}$, where $\CC$ is a skeleton of
$\CA_\lambda^\lambda$, and put $C_1=\coprod_{(U,h)\in\CN}U$ with
the canonical injections $k_{(U,h)}:U\to C_1$. By the choice of
$\lambda$ we have $\lambda>\card\CA(U,C_0)\geq\card\CN$, hence
$(C_1,0)\in\CA^\lambda_\lambda/F$. We may even consider
$(C_1,0)\in\CC/F$. We have $(U_i,(f_1-f_2)j_i)\in\CN$, hence there
is a unique $\theta:A\to C_1$ such that $\theta
j_i=k_{(U_i,(f_1-f_2)j_i)}$ for all $i\in I$. Further there is a
unique morphism $\eta:C_1\to C_0$ such that $\eta
k_{(U,h)}=h$ for all $(U,h)\in\CN$. Clearly $\eta$
is a map in $\CA_\lambda^\lambda/F$ between $(C_1,0)$ and
$(C_0,c_0)$. If $C$ is defined by the exactness of the sequence
$C_1\stackrel{\eta}\longrightarrow
C_0\stackrel{\delta}\longrightarrow C\to 0$, then
$C\in\CA_\lambda^\lambda$, because $\CA_\lambda^\lambda$ is closed
under cokernels (see \cite[Proposition 1.16]{AR}). Since $F$ sends
cokernels into kernels, we infer that there is $c\in F(C)$ such
that $F(\delta)(c)=c_0$. Thus $\delta:(C_0,c_0)\to(C,c)$ lies in
$\CA_\lambda^\lambda/F$, and $\delta(f_1-f_2)=\delta\eta\theta=0$,
finishing the proof of the first step above.

Finally an arbitrary $A\in\CA_\kappa$ is a colimit of objects in
$\CA_\kappa^\kappa$, so it is a cokernel of the form $A_1\to
A_0\to A\to 0$ with $A_1$ and $A_0$ being coproducts of objects in
$\CA_\kappa^\kappa$. Using the first step before, we get easily
\[F(A)\cong\CA(A,F_\lambda)\cong\colim_{(X,x)\in\CA_\lambda^\lambda/F}\CA(A,X)\]
canonically.
\end{proof}

\begin{remark}\label{franke}
With the notations made in Lemma \ref{fiscolimit} and its proof,
the argument used to show the fact that
$\CA(A,F_\lambda)\cong\colim_{(X,x)\in\CA_\lambda^\lambda/F}\CA(A,X)$,
for $A=\coprod_{i\in I}U_i$, with $U_i\in\CA_\kappa^\kappa$ is
inspired by \cite[Lemma 2.11]{F}. However, we didn't only change
the settings, but we also improved the proof of Franke. A simple
translation of his argument in our settings would require the
condition $\card\CA(U,X)\leq\lambda$ for all
$U\in\CA_\kappa^\kappa$ and all $X\in\CA^\lambda_\lambda$. A
priori is not clear how we may choose such a regular cardinal
$\lambda$. Instead this, we required
$\sum_{U\in\CC_0}\card\CA(U,C_0)<\lambda$, where the left hand
side of this inequality doesn't depend of $\lambda$.
\end{remark}

Recall that we call {\em cofinal} a subcategory $\CS$ of a
category $\CC$ satisfying the following two properties: For every
$c\in\CC$ there is a map $c\to s$ in $\CC$ for some $s\in\CS$; and
for any two maps $c\to s_1$ and $c\to s_2$ in $\CC$, with
$s_1,s_2\in\CS$ there are $s\in\CS$ and two maps $s_1\to s$ and
$s_2\to s$ in $\CS$ such that the composed morphisms $c\to s_1\to
s$ and $c\to s_2\to s$ are equal.  It is well--known that if $\CS$
is a cofinal subcategory of $\CC$, then colimits over $\CC$ and
colimits over $\CS$ coincide (see \cite[0.11]{AR}).

\begin{lemma}\label{cws} Let $\CA$ be an abelian category, and let
$F:\CA\to\Ab$ be a contravariant, exact functor. Let
$\CC\subseteq\CA$ be a subcategory closed under finite coproducts
and cokernels. If $\CS$ is a subcategory of $\CC$ closed under
finite coproducts and satisfying the property that every $X\in\CC$
admits an embedding $0\to X\to S$ into an object in $\CS$, then
$\CS/F$ is a cofinal subcategory of $\CC/F$.
\end{lemma}

\begin{proof}
Let $(X,x)\in\CC/F$. Consider an embedding  $0\to
X\stackrel{\alpha}\to S$, with $S\in\CS$. Thus
$F(S)\stackrel{F(\alpha)}\longrightarrow F(X)\to 0$ is exact,
showing that there exists $y\in F(S)$ with $F(\alpha)(y)=x$.
Therefore $\alpha$ is a map in $\CC/F$ between $(X,x)$ and
$(S,y)$.

Now we claim that if $\alpha:X_1\to X_2$ is a map in $\CC$, and
$x_2\in F(X_2)$ is an element with the property
$F(\alpha)(x_2)=0$, then there is a morphism
$\gamma\in\CC/F((X_2,x_2),(S,y))$ into an object $(S,y)\in\CS/F$
such that $\gamma\alpha=0$. Indeed consider $X$ being defined by
exact sequence $X_1\stackrel{\alpha}\to X_2\stackrel{\beta}\to
X\to 0$. Since the sequence of abelian groups $0\to
F(X)\stackrel{F(\beta)}\to F(X_2)\stackrel{F(\alpha)}\to F(X_1)$
is also exact and $F(\alpha)(x_2)=0$, we obtain an element $x\in
F(X)$ such that $F(\beta)(x)=x_2$. For obtaining the required
$\gamma$, compose $\beta$ with a morphism in $\CC/F$ from $(X,x)$
into an object $(S,y)$, which is constructed as in the first part
of this proof.

Finally for two morphisms
\[\alpha_1\in\CC/F((X,x),(S_1,y_1))\hbox{ and }
\alpha_2\in\CC/F((X,x),(S_2,y_2)),\] denote by $\rho_1$ and
$\rho_2$ the respective injections of the coproduct $S_1\amalg
S_2$. Then $F(\rho_1\alpha_1-\rho_2\alpha_2)(y_1,y_2)=x-x=0$, so
our claim for $\alpha=\rho_1\alpha_1-\rho_2\alpha_2$ gives a
morphism $(S_1\amalg S_2, (y_1,y_2))\to (S,y)$ in $\CC/F$, with
$S\in\CS$, such that the composed morphisms $X\to S_1\to S_1\amalg
S_2\to S$ and $X\to S_2\to S_2\amalg S_2\to S$ are equal.
\end{proof}

Let $\kappa\in\GR$. As usually, a {\em $\kappa$--(co)product}
means a (co)product of less that $\kappa$ objects. We say that a
quasi--locally presentable abelian category $\CA$ is {\em weakly
$\kappa$--generated} if $\CA$ coincide with its smallest full
subcategory containing $\CA_\kappa$ and being closed under
kernels, cokernels, extensions and $\kappa$--coproducts. We also
need the following notation:
\[\Inj_\lambda\CA=\{S\in\CA\mid S\hbox{ is injective and
}S\in\CA_\lambda^\lambda\}.\]

\begin{theorem}\label{represent}
Let $\CA$ be a quasi--locally presentable, abelian category which
is weakly $\kappa$--generated, for some regular cardinal $\kappa$.
Suppose also that, for any regular cardinal $\lambda\geq\kappa$,
every $X\in\CA_{\lambda}^{\lambda}$ admits an embedding $0\to X\to
S$ into an object $S\in\Inj_\lambda\CA$. Then every exact,
contravariant functor $F:\CA\to\Ab$ which sends coproducts into
products is representable (necessarily by an injective object).
\end{theorem}

\begin{proof} Fix a contravariant exact functor $F:\CA\to\Ab$, which sends
coproducts into products. Consider the obvious natural transformation
\[\phi:\colim_{(X,x)\in\CA^\lambda_\lambda/F}\CA(-,X)\to F.\]
Since $F$ sends colimits into limits, Lemma \ref{fiscolimit}
applies and tells us that there is $\lambda\in\GR$,
$\lambda\geq\kappa$ such that $\phi$ restricts to an isomorphism:
\[\colim_{(X,x)\in\CA^\lambda_\lambda/F}\CA(I_\kappa(-),X)\cong FI_\kappa.\]
We know that $\CA^\lambda_\lambda/F$ is $\lambda$--filtered (see
\cite[Korollar 5.4]{GU}), hence colimits of abelian groups indexed
over this category are exact and commute with products of less
that $\lambda$ objects (see \cite[Satz 5.2]{GU}). Since every
$X\in\CA_\lambda^\lambda$ admits an embedding in an object
$S\in\Inj_\lambda\CA$, we deduce by Lemma \ref{cws} that
$\Inj_\lambda\CA/F$ is a cofinal subcategory of
$\CA^\lambda_\lambda/F$, so
\[\colim_{(X,x)\in\CA^\lambda_\lambda/F}\CA(-,X)\cong\colim_{(S,s)\in\Inj_\lambda\CA/F}\CA(-,S)\]
is an exact functor. We infer that the full subcategory of $\CA$
consisting of all objects $A$ for which $\phi_A$ is an isomorphism
contains $\CA_\kappa$ and is closed under kernels, cokernels,
extensions and $\kappa$-coproducts (since $\lambda\geq\kappa$).
Therefore it is equal to $\CA$ forced by the hypothesis of weak
$\kappa$--generation. This means that $\phi$ is a natural
isomorphism, hence a skeleton of $\CA_\lambda^\lambda$ forms a
solution set for $F$. We conclude that $F$ is representable by the
general Freyd's adjoint functor theorem.
\end{proof}

\begin{example}\label{locG}
The following example shows that the conclusion of Theorem
\ref{represent} requires a kind of weak generation.

Recall that an abelian category is called
{\em locally Grothendieck} if every set of objects may be included
in subcategory which is Grothendieck (see \cite{T}). 
Let $K$ be a field. The category
$\CA=\bigcup_{\lambda\in\GR}\Md(K^\lambda)$ considered in \cite{T}
is locally Grothendieck. Here by $\Md(K^\lambda)$ we denote the
category of right modules over the ring $K^\lambda$. Moreover, the category
$\CA$ is also quasi--locally presentable. Indeed it is a a big union of
a chain of Grothendieck (hence locally presentable) subcategories
$\CA_\lambda=\Md(K^\lambda)$. For all $\kappa\leq\lambda$
in $\GR$ we have $K^\kappa=K^\lambda e$, where
$e=e({\kappa,\lambda})\in K^\lambda$ is a central idempotent
defined by $e_{\gamma}=1$ for $\gamma\leq\kappa$ and $0$
otherwise. Thus $K^\kappa$ is a direct summand of $K^\lambda$, and
all $X\in\Md(K^\lambda)$ decomposes as $X=Xe\oplus X(1-e)$.
Moreover for $X,Y\in\Md(K^\lambda)$ there is no nonzero
homomorphisms between $Xe$ and $Y(1-e)$, hence we have
\[\Hom_{K^\lambda}(X,Y)=\Hom_{K^\kappa}(Xe,Ye)\oplus\Hom_{K^\lambda(1-e)}(X(1-e),Y(1-e)).\]
Thus we can see $\Md(K^\kappa)$ as a full split subcategory of
$\Md(K^\lambda)$. 
We deduce that for every fixed $\kappa\in\GR$ and for every $X\in\CA$, there is $\lambda\geq\kappa$
such that $X\in\Md(K^\lambda)$. The assignment $X\mapsto Xe$, where
$e=e(\kappa,\lambda)$ induces a well defined functor $R_\kappa:\CA\to\Md(K^\kappa)$ which is both
the left and the right adjoint of the inclusion functor $I_\kappa$; this follows by the fact that
$\Md(K^\kappa)$ is a full split subcategory of $\Md(K^\lambda)$. Thus both the inclusion functor
$\Md(K^\kappa)$ and its right adjoint preserve colimits. 

Using an idea from \cite{MS} we may construct a non--representable
exact contravariant functor $F:\CA\to\Ab$, which sends coproducts
into products. For every $\lambda\in\GR$, denote by $\lambda^+$
the successor of $\lambda$ and consider $Q_{\lambda^+}$ to be an
injective cogenerator of $\Md(K^{\lambda^+})$. The
$K^{\lambda^+}$-module $Y_\lambda=Q_{\lambda^+}(1-e)$, where
$e=e(\lambda,\lambda^+)$, is injective and satisfies
$\Hom_{K^{\lambda^+}}(X,Y_\lambda)=0$ for all
$X\in\Md(K^\lambda)$. The contravariant functor
\[F:\CA\to\Ab, F(X)=\prod_{\lambda\in\GR}\CA(X,Y_{\lambda})\] is
well defined. In fact, for $X\in\Md(K^\kappa)$, we have
$\CA(X,Y_\lambda)=0$ if $\lambda\geq\kappa$, hence
$F(X)=\prod_{\lambda<\kappa}\CA(X,Y_{\lambda})$. Obviously $F$ is
exact and sends coproducts into products. But $F$ is not
representable, since the strict inclusion of $\Md(K^\lambda)$ into
$\Md(K^{\lambda^+})$ implies that the cogenerator $Q_{\lambda^+}$
must contain a nonzero part $Y_\lambda$ in
$\Md(K^{\lambda^+}(1-e))$. The representability of $F$ would means
the existence of the product $Y=\prod_{\lambda\in\GR}Y_\lambda$ in
$\CA$. But this is absurd since $Y$ would have a proper class of
endomorphisms, and such objects don't exist in $\CA$. Notice that
the category $\bigcup_{\lambda\in\GR}\Md(K^\lambda)$ was used in
\cite{T} as example of a category for which the $\lambda$-pure
global dimension is greater that 1, for all $\lambda\in\GR$;
both this example and our present work have connections with Brown
representability. On the other hand we have:

\begin{proposition}
Consider the above locally Grothendieck category \[\CA=\bigcup_{\lambda\in\GR}\Md(K^\lambda).\]
A contravariant functor $F:\CA\to\Ab$ is representable if and only if it sends colimits
into limits and there is $\kappa\in\GR$ such that $F\cong FI_\kappa R_\kappa$.
\end{proposition}

\begin{proof} 
If $F\cong\CA(-,Y)$ for some $Y\in\CA$ then there is $\kappa\in\GR$ such that
$Y\in\Md(K^\kappa)$. Thus for every $X\in\CA$, there is $\lambda\geq\kappa$ such that $X\in\Md(K^\lambda)$, hence
$F(X)=\CA(X,Y)\cong\CA(Xe,Y)\cong FI_\kappa R_\kappa(X).$

Conversely if $F$ sends colimits into limits then, as in the proof of Lemma \ref{fiscolimit},
we obtain $FI_\kappa\cong\Hom_{K^\kappa}(-,Y)$, for some $Y\in\Md(K^\kappa)$. Combining this
with $F\cong FI_\kappa R_\kappa$ we deduce:
\[F\cong\Hom_{K^\kappa}(R_\kappa(-),Y)\cong\CA(-,I_\kappa(Y)),\] therefore $F$ is representable.
\end{proof}
\end{example}

\begin{example}\label{norep} In Theorem \ref{represent} the
exactness of the functor $F:\CA\to\Ab$ (which sends coproducts
into products) is an essential hypothesis. More precisely, the
weaker requirement that $F$ sends colimits into limits is not
sufficient to conclude that it is representable.  For showing this,
suppose that the quasi--locally presentable category $\CA$ from
the Theorem \ref{represent} is abelian (as in the motivating case of the next Section) 
but is not locally presentable, that is
$\CA\neq\CA_\lambda$ for every $\lambda\in\GR$. The fact that
$\CA$ is weakly generated which is used in combination with the
exactness of $F$ doesn't play any role in this example. The exactness of 
$R_\lambda$ implies that  $\CA_\lambda$ is equivalent to quotient category of $\CA$ 
modulo the Serre subcategory $\Ker R_\lambda=\{X\in\CA\mid R_\lambda(X)=0\}$. But $R_\lambda$  
is not an equivalence, forcing $\Ker R_\lambda\neq0$. Consider $0\neq
X_\lambda\in\CA$ such that $R_\lambda(X_\lambda)=0$, for every
$\lambda\in\GR$. Strictly speaking we need here a version of axiom
of choice which works for proper classes.  As in Example \ref{locG},
we infer that the functor
\[F=\prod_{\lambda\in\GR}\CA(-,X_\lambda)\] is
well defined since for every $X\in\CA$ we have $X\in\CA_\kappa$
for some $\kappa\in\GR$, so $\CA(X,X_\lambda)=0$ for all
$\lambda\geq\kappa$. It is easy to see that this functor does the
job we claim.
\end{example}

\section{The abelianization of a well generated triangulated
category}

The main purpose of this section is to show that the
abelianization of a triangulated category which is well generated
in the sense of Neeman is  quasi--locally presentable and
satisfies the hypothesis of Theorem \ref{represent}. Consequently
we obtain a new proof of Brown representability theorem for such
triangulated categories.

Consider a preadditive category $\CT$.  By a {\em$\CT$-module\/}
we understand a functor $X:\CT\opp\to\Ab$. Such a functor is
called {\em finitely presentable\/} if there is an exact sequence
of functors \[\CT(-,y)\to\CT(-,x)\to X\to0\] for some $x,y\in\CT$.
Using Yoneda lemma, we know that the class of all natural
transformations between two $\CT$-modules $X$ and $Y$ denoted
$\Hom_\CT(X,Y)$ is actually a set, provided that $X$ is finitely
presentable. We consider the category $\md(\CT)$ of all finitely
presentable $\CT$-modules, having $\Hom_\CT(X,Y)$ as morphisms
spaces, for all $X,Y\in\md(\CT)$. The Yoneda functor
\[H=H_\CT:\CT\to\md(\CT)\hbox{ given by }H_\CT(x)=\CT(-,x)\] is  an
embedding of $\CT$ into $\md(\CT)$, according to Yoneda lemma. If,
in addition, $\CT$ has coproducts then $\md(\CT)$ is cocomplete
and the Yoneda embedding preserves coproducts. It is also
well--known (and easy to prove) that, if $F:\CT\to\CA$ is a
functor into an additive category with cokernels, then there is a
unique, up to a natural isomorphism, right exact functor
$F^*:\md(\CT)\to\CA$, such that $F=F^*H_\CT$ (see \cite[Lemma
A.1]{KL}). Moreover, $F$ preserves coproducts if and only if $F^*$
preserves colimits.

In this section the category $\CT$ will be triangulated with
splitting idempotents. For definition and basic properties of
triangulated categories the standard reference is \cite{N}. Note
that $\CT$ has splitting idempotents, provided that $\CT$ has
countable coproducts, according to \cite[Proposition 1.6.8]{N}.
Recall that $\CT$ is supposed to be additive. A functor
$\CT\to\CA$ into an abelian category $\CA$ is called {\em
homological\/} if it sends triangles into exact sequences. A
contravariant functor $\CT\to\CA$ which is homological regarded as
a functor $\CT\opp\to\CA$ is called {\em cohomological\/} (see
\cite[Definition 1.1.7 and Remark 1.1.9]{N}). An example of a
homological functor is the Yoneda embedding
$H_\CT:\CT\to\md(\CT)$. We know: $\md(\CT)$ is an abelian
category, and  for every functor $F:\CT\to\CA$  into an abelian
category, the unique right exact functor $F^*:\md(\CT)\to\CA$
extending $F$ is exact if and only if $F$ is homological, by
\cite[Lemma 2.1]{KS}.  This is the reason for which $\md(\CT)$ is
called the {\em abelianization} of the triangulated category $\CT$
and is denoted sometimes by $\A(\CT)$. By \cite[Corollary
5.1.23]{N}, $\A(\CT)$ is a Frobenius abelian category, with enough
injectives and enough projectives, which are, up to isomorphism,
exactly objects of the form $\CT(-,x)$ for some $x\in\CT$.

A first link between representability of functors defined on
$\CT$, respectively on $\A(\CT)$ is given by:

\begin{lemma}\label{fstar} If $\CT$ is a triangulated category
with splitting idempotents, then a cohomological functor
$F:\CT\to\Ab$ is representable if and only if its extension
$F^*:\A(\CT)\to\Ab$ is representable.
\end{lemma}

\begin{proof} The cohomological functor $F:\CT\to\Ab$
can be interpreted as a  homological
functor $\CT\to\Ab\opp$ which has a unique extension to $\A(\CT)$.
Therefore $F$ extends uniquely to a contravariant, exact functor
$F^*:\A(\CT)\to\Ab$, defined as $F^*\cong\Hom_\CT(-,F)$. We recall
that $\Hom$ denotes the set of all natural transformations, and it
coincides with the morphisms spaces in $\A(\CT)$ only if
$F\in\A(\CT)$.

If $F$ is representable, then $F\in\A(\CT)$, and $F^*$ is
represented by $F$. Conversely if $F^*$ is representable by an
object in $\A(\CT)$ then this object must be isomorphic to $F$,
therefore $F\in\A(\CT)$. Because $F^*$ is exact, $F$ must be
injective, hence representable.
\end{proof}

We say that $\CT$ satisfies Brown representability theorem if
every cohomological functor $F:\CT\to\Ab$ which sends coproducts
into products is representable. Then we record:

\begin{corollary}\label{reform}
Let $\CT$ be a triangulated category with coproducts. The
following are equivalent: \begin{itemize} \item[{\rm (i)}] $\CT$
satisfies Brown representability theorem. \item[{\rm (ii)}] Every
exact contravariant functor $F:\A(\CT)\to\Ab$ which sends coproducts into products
 is representable. \item[{\rm (iii)}] Every exact
covariant functor $F:\A(\CT)\to\CA$ which preserves
colimits, having values into an abelian cocomplete category with
enough injectives has a right adjoint.
\end{itemize}
\end{corollary}

\begin{proof} The equivalence (i)$\Leftrightarrow$(ii) follows by
Lemma \ref{fstar}, whereas and the implication
(iii)$\Rightarrow$(ii) is obvious, by replacing contravariant
functors $\A(\CT)\to\Ab$ with covariant functors
$\A(\CT)\to\Ab\opp$. Finally (i)$\Rightarrow$(iii) follows by
\cite[Theorem 1.1]{BM}.
\end{proof}

Let $\CT$ is a triangulated category with coproducts. We need the
following definitions: For regular cardinal $\lambda$, a
{\em$\lambda$--localizing} subcategory of $\CT$ is a triangulated
subcategory closed under $\lambda$--coproducts. A {\em localizing}
subcategory is a subcategory which is $\lambda$--localizing, for
all $\lambda$. Consider a set of objects $\CS\subseteq\CT$ which is closed under suspensions 
and desuspensions. We say that $\CT$ is generated (in the triangulated
sense) by $\CS$, provided that an
object $t\in\CT$ vanishes, whenever $\CT(s,t)=0$ for all
$s\in\CS$. Further we say that $\CT$ is perfectly generated by the
set of objects $\CS$ if $\CS$ generates $\CT$ and, for any
$s\in\CS$, the map $\CT(s,\coprod_{i\in
I}x_i)\to\CT(s,\coprod_{i\in I}y_i)$ is surjective, for every set
of maps $\{x_i\to y_i\mid i\in I\}$ such that
$\CT(s,x_i)\to\CT(s,y_i)$ is surjective, for all $i\in I$. Finally
$\CT$ is called {\em well $\lambda$--generated}, where
$\lambda\in\GR$, provided that $\CT$ is perfectly generated by a
set of objects which are also {\em $\lambda$--small}, that is,
every map $s\to\coprod_{i\in I}x_i$, with $s\in\CS$, factors
trough a coproduct $\coprod_{i\in I'}x_i$ with $\card
I'<\lambda$; the category $\CT$ is well generated if it is well
$\lambda$--generated, for some $\lambda$. Following \cite[Theorem
A]{KN}, this definition is equivalent to the original one given by
Neeman. Note that, by \cite[Corollary 2.6]{MB}, if $\CT$ is
perfectly generated by $\CS$, then $\CT$ coincides with its
smallest $\aleph_1$--localizing subcategory which contains
arbitrary coproducts of objects in $\CS$.

A category $\CC$ is called {\em$\lambda$--cocomplete} if
$\CC$ has $\lambda$--coproducts and cokernels. It is easy to see
that $\CC$ is $\lambda$--cocomplete if and only if it contains all
colimits of diagrams with less that $\lambda$ morphisms. A
$\CC$-module over a $\lambda$--cocomplete category is called {\em
$\lambda$--left exact\/} if it is left exact and sends
$\lambda$--coproducts into products. Provided that the category
$\CC$ is essentially small, the class $\Hom_\CC(X,Y)$ is actually
a set for all $\CC$-modules $X,Y$. Thus we are allowed to consider the
category $\Md(\CC)$ of all $\CC$-modules. If $\CC$ is also
$\lambda$--cocomplete, then denote by $\lex_\lambda(\CC^\opp,\Ab)$
the full subcategory of $\Md(\CC)$ consisting of $\lambda$--left
exact modules. We know that $\lex_\lambda(\CC^\opp,\Ab)$ is a
locally $\lambda$--presentable category, and the embedding
$\CC\to\lex_\lambda(\CC^\opp,\Ab)$ given by $X\mapsto\CC(-,X)$
identifies $\CC$, up to isomorphism, with the subcategory of
$\lambda$--presentable objects in $\lex_\lambda(\CC^\opp,\Ab)$
(see \cite[Korollar 7.9]{GU}).

As before, let $\lambda$ denote a regular cardinal.  If $\CS$ is
an preadditive, essentially small category with
$\lambda$--coproducts, denote by $\Prod_\lambda(\CS^\opp,\Ab)$ the
full subcategory of $\Md(\CS)$, consisting of those modules which
preserve $\lambda$--products. Clearly a finitely presentable
$\CS$-module, that is an element in $\md(\CS)$, preserves
arbitrary products, hence it belongs to
$\Prod_\lambda(\CS\opp,\Ab)$.

\begin{lemma}\label{lexeqpr} For a regular cardinal $\lambda$,
consider an additive, essentially small category $\CS$ having
$\lambda$--coproducts.  Then $\Prod_\lambda(\CS^\opp,\Ab)$ is a
locally $\lambda$--presentable category, and the embedding
$\md(\CS)\stackrel{\subseteq}\longrightarrow\Prod_\lambda(\CS^\opp,\Ab)$
identifies $\md(\CS)$ with the full subcategory of
$\Prod_\lambda(\CS^\opp,\Ab)$ consisting of all
$\lambda$--presentable objects.
\end{lemma}

\begin{proof} The category $\md(\CS)$ has obviously
$\lambda$--coproducts and cokernels, so it is
$\lambda$--cocomplete.  According to \cite[Lemma B.1]{KL}, there
is an equivalence of categories
\[\lex_\lambda(\md(\CS)^\opp,\Ab)\to\Prod_\lambda(\CS^\opp,\Ab),\
X\mapsto XH_\CS,\] where $H_\CS:\CS\to\md(\CS)$ denotes the Yoneda
functor. Thus $\Prod_\lambda(\CS\opp,\Ab)$ is locally
$\lambda$--presentable. Further, the identification of
$\lambda$--presentable objects in $\Prod_\lambda(\CS\opp,\Ab)$
follows by discussion above concerning $\lambda$--presentable
objects in $\lex_\lambda(\CC\opp,\Ab)$. \end{proof}

Suppose now that $\CT$ is well $\kappa$--generated triangulated
category, having a perfectly generating set $\CS$ consisting of
$\kappa$-small objects. For any $\lambda\geq\kappa$ we consider
the smallest $\lambda$ localizing subcategory of $\CT$ which
contains $\CS$ and denote it by $\CT^\lambda$. The objects in $\CT^\lambda$
are called {\em$\lambda$--compact}. By \cite[Lemma
5]{KN} the category of $\lambda$--compact objects in $\CT$ is
independent of $\CS$. Clearly it is essentially small and a
skeleton of $\CT^\lambda$ generates $\CT$. Moreover $\CT^\lambda$ has $\lambda$-coproducts.
Denote $\A_\lambda(\CT)=\Prod_\lambda((\CT^\lambda)\opp,\Ab)$, for
$\lambda\geq\kappa$ and $\A_\lambda(\CT)=0$ otherwise. We know by
\cite[Proposition A.1.8]{N} that $\A_\lambda(\CT)$ is locally
$\lambda$--presentable, and by \cite[Proposition 6.5.3]{N} that 
the restriction functor $R_\lambda:\A(\CT)\to\A_\lambda(\CT)$ has a fully faithful 
left adjoint $I_\lambda:\A_\lambda(\CT)\to\A(\CT)$, therefore we may identify 
$\A_\lambda(\CT)$ to a coreflective subcategory of $\A(\CT)$.

\begin{proposition}\label{wellgen}
Fix a regular cardinal $\kappa>\aleph_0$. If $\CT$ is a well
$\kappa$--generated triangulated category, then $\A(\CT)$ is a
quasi--locally presentable abelian category which is weakly
$\kappa$--generated.
\end{proposition}

\begin{proof}
Denote by $\CA$ the smallest subcategory of $\A(\CT)$ which is
closed under kernels, cokernels, extensions, countable coproducts
and contains $\A_\kappa(\CT)$. Let us show that $\A(\CT)=\CA$.
Observe first that if $T\to U\to X\to Y\to Z$ is an exact sequence
with $T,U,Y,Z\in\CA$ then we can construct the commutative diagram
with
exact rows and column \[\diagram &&0\dto&&\\
T\ddouble\rto&U\ddouble\rto&X'\rto\dto&0&\\
T\rto&U\rto&X\rto\dto&Y\rto\ddouble&Z\ddouble\\
&0\rto&X''\rto\dto&Y\rto&Z\\
&&0&&
\enddiagram\] showing that $X\in\CA$. Therefore if $x\to y\to z\rightsquigarrow$
is a triangle in $\CT$ with $H(x),H(z)\in\CA$ then $H(y)\in\CA$.
It is shown in \cite[Theorem 2.5]{MB} that every object $x\in\CT$
is isomorphic to a homotopy colimit of a tower $x^0\to
x^1\to\cdots$ such that $x^0=0$ and for every $n\in\N$ we have a
triangle $p_n\to x^n\to x^{n+1}\rightsquigarrow$ with $p_n$ being
a coproduct of objects in $\CT^\kappa$. Inductively
$H(x^n)\in\CA$, for all $n\in\N$, hence
$H(\coprod_{n\in\N}x^n)\cong\coprod_{n\in\N}H(x^n)\in\CA$, and
finally $H(x)\in\CA$. Now, for every $X\in\A(\CT)$ there is an
exact sequence $H(y)\to H(x)\to X\to 0$, with $x,y\in\CT$, thus
$X\in\CA$.

Note that we have already shown that $\CT$ coincides with its
smallest $\aleph_1$-localizing subcategory which contains a
skeleton of $\CT^\kappa$. Therefore the proof of \cite[Proposition
8.4.2]{N} (more precisely \cite[8.4.2.3]{N}) works for our case,
hence $\CT=\bigcup_{\lambda\geq\kappa}\CT^\lambda$, and further 
$\A(\CT)=\bigcup_{\lambda\in\GR}\A_\lambda(\CT)$. In addition an immediate consequence of 
Lemma \cite[6.5.1]{N} is that the right adjoint of the inclusion functor 
$\A_\lambda(\CT)\to\A(\CT)$ preserves colimits, and all conditions from the definition of a 
weakly $\kappa$-generated quasi--locally presentable category are fulfilled. 
\end{proof}

\begin{theorem}\label{rep}
If $\CT$ is a well generated triangulated category, then every
functor $F:\A(\CT)\to\Ab$ which is contravariant, exact and sends
coproducts into products is representable.
\end{theorem}

\begin{proof} Without losing the generality we may assume that
$\CT$ is well $\kappa$--generated, for some $\kappa\geq\aleph_1$
(if not, we replace $\kappa$ by $\aleph_1$). By Proposition
\ref{wellgen}, $\A(\CT)$ is a weakly $\kappa$--generated
quasi--locally presentable category. In order to apply Theorem
\ref{represent} we have only to show that every
$\lambda$--presentable object $X$ of $\A_\lambda(\CT)$ admits an
embedding into an object in $S\in\A_\lambda(\CT)$ which is
$\lambda$--presentable in $\A_\lambda(\CT)$ and injective in
$\A(\CT)$. But this follows immediately from Lemma \ref{lexeqpr},
since, according to \cite[Corollary 5.1.23]{N}, every
$X\in\md(\CT^\lambda)$ admits an embedding into an object of the
form $H(x)$ with $x\in\CT^\lambda$.
\end{proof}

Note that the category $\A(\CT)$ is usually ``huge'', in the sense
that it is not well (co)powered, as we learned on \cite[Appendix
C]{N}. Thus Proposition \ref{wellgen} and Theorem \ref{rep}
provide an example of such a huge category which is quasi--locally
presentable and for which representability Theorem \ref{represent}
applies.

Combining Theorem \ref{rep} with Corollary \ref{reform} we obtain
a new proof for:

\begin{corollary}\label{brt}
Every well generated triangulated categories satisfies Brown
representability theorem.
\end{corollary}

\begin{example}\label{lwgnorep} Recall from \cite{S} the definition:
A triangulated category with coproducts is called {\em locally
well generated}, provided that every localizing subcategory which
is generated (in the triangulated sense) by a set of objects is
well generated. The typical example of a locally well generated
triangulated category, which is not well generated, is the
homotopy category $\HK(\Md R)$ where $R$ is a ring which is not
pure--semisimple (see \cite[Theorem 3.5]{S}). Objects in this
category are complexes of $R$-modules, and maps are classes of
homotopy equivalent maps of complexes.

Let consider $R=\Z$, so $\CT=\HK(\Ab)$ is locally well generated,
but not well generated.  Then we want to construct a non--representable exact
contravariant functor $F:\A(\HK(\Ab))\to\Ab$, which sends
coproducts into products. For this purpose, observe that
there are objects $Y_\lambda\in\HK(\Ab)$
with $\lambda\in\GR$ such that the functor:
\[F=\prod_{\lambda\in\GR}\CT(-,Y_\lambda):\HK(\Ab)\to\Ab\]
is cohomological, sends coproducts into products but is not
representable, as it may be seen in \cite[Example 11]{MS}. Note that the argument
showing that this functor is well defined is similar to the one used in Examples \ref{locG} and \ref{norep}.
By Lemma \ref{fstar} the functor
\[F^*:\A(\HK(\Ab))\to\Ab, F^*(X)=\Hom_{\HK(\Ab)}\left(X,\prod_{\lambda\in\GR}\CT(-,Y_\lambda)\right)\] is
contravariant, exact, sends coproducts into products, but is not
representable.
\end{example}


\begin{thebibliography}{99}


\bibitem{AR} J. Ad\'amek and J. Rosick\'y
{\em Locally presentable and accessible categories\/}, Cambridge
University Press, 1994.

\bibitem{BM} S. Breaz, G. C. Modoi, {\em A reformulation of Brown representability
theorem}, Mathematica(Cluj), {\bf51}(2009), 129--133.


\bibitem{F} J. Franke,  {\em On the Brown representability theorem for triangulated
categories\/}, Topology, {\bf40}(2001), 667--680.

\bibitem{PF} P. Freyd, {\em Abelian Cayegories. An Introduction to the Theory of
Functors}, Harper \& Row, New York, 1964.

\bibitem{GU} P. Gabriel, F. Ulmer, {\em Lokal pr{\"a}sentierbare
    Kategorien\/}, Springer Lecture Notes in Math.,  {\bf221}, Berlin--Heidelberg, 1971.

\bibitem{KS} H. Krause, {\em Smashing subcategories and the telescope conjecture
-- an algebraic approach\/}, {Invent. Math.} {\bf139}(2000),
99--133.

\bibitem{KN} H. Krause, {\em On Neeman's well generated triangulated
categories\/},  Documenta Math. {\bf6}(2001), 121--126.


\bibitem{KL} H. Krause, {\em Localization theory for triangulated
categories\/},  to appear in the proceedings of the {\em Workshop
on Triangulated Categories}, Leeds 2006.


\bibitem{MB} G. C. Modoi, {\em On perfectly generating projective classes in
triangulated categories}, Comm. Algeba, 38(2010), 995--1011.

\bibitem{MS} G. C. Modoi, J. \v S\v tov\'\i \v cek,
{\em Brown representability often fails for homotopy categories of
complexes}, J. K-Theory, to appear.

\bibitem{N} A. Neeman, {\em Triangulated Categories\/}, Annals of
  Mathematics Studies, {\bf148}, Princeton University Press,
  Princeton, NJ, 2001.




\bibitem{S} J. \v S\v tov\'\i \v cek, {\em Locally well generated homotopy categories of
complexes}, Doc. Math. 15(2010), 507--525.

\bibitem{T} J. Trlifaj, {\em Brown representability test problems for locally Grothendieck
categories}, Appl. Cat. Struct., to appear, DOI:
10.1007/s10485-010-9234-z.

\end{thebibliography}
\end{document}